\documentclass[11pt,a4paper,UKenglish]{article}
\usepackage[utf8]{inputenc}
\usepackage[english]{babel}
\usepackage{scrpage}
\usepackage{alltt} 
\usepackage{amsmath}
\usepackage{amsthm}
\usepackage{amsfonts}
\usepackage{amssymb}
\usepackage{txfonts}
\usepackage{mathrsfs}
\usepackage{verbatim}
\usepackage{lmodern}
\usepackage[T1]{fontenc}
\usepackage[arrow, matrix, curve]{xy}
\newtheorem{thm}{Theorem}[section]

\newtheorem{lem}[thm]{Lemma}
\newtheorem{cor}[thm]{Corollary}
\newtheorem{prop}[thm]{Proposition}
\newtheorem*{theorem}{Theorem}
\theoremstyle{definition}
\newtheorem*{rem}{Remark}
\newtheorem*{rems}{Remarks}

\newtheorem*{thanksto}{Acknowledgements}
\mathchardef\mhyphen="2D
\title{Orlov's Equivalence and Maximal Cohen-Macaulay Modules over the Cone of an Elliptic Curve}
\author{Lennart Galinat}
\date{}
\begin{document}
\thispagestyle{empty}
\maketitle
\begin{abstract}
	We describe a method for doing computations with Orlov's equivalence between the bounded derived category of certain hypersurfaces and the stable category of graded matrix factorisations of the polynomials describing these hypersurfaces. In the case of a smooth elliptic curve over an algebraically closed field we describe the indecomposable graded matrix factorisations of rank one. Since every indecomposable Maximal Cohen-Macaulay module over the completion of a smooth cubic curve is gradable, we obtain explicit descriptions of all indecomposable rank one matrix factorisations of smooth cubic potentials. Finally, we explain how to compute all indecomposable matrix factorisations of higher rank with the help of a computer algebra system.
\end{abstract}
\thispagestyle{empty}
\addcontentsline{toc}{section}{Introduction}
\tableofcontents
\setcounter{page}{1}
\section*{Introduction}
	Let $K$ be a field, let $f \in K[X_0, \hdots, X_n]$ be a homogeneous polynomial of degree $n+1$ and let $X = \mathsf{Proj}\Big(K[X_0, \hdots, X_n]/(f)\Big)$ be the projective scheme associated to $f$. By a theorem of Orlov \cite{Orlov09} there is an equivalence $$\Phi: \mathcal{D}^b(\mathsf{Coh}(X)) \cong \underline{\mathsf{MF}}(f) $$ between the bounded derived category of coherent sheaves on $X$ and the homotopy category of graded matrix factorisations of $f$. It is only natural to want to use this to transfer questions about objects on one side to the other in the hope of finding an easier answer there. Therefore we might ask the following questions:
	\begin{enumerate}
		\item What is the action of $\mathsf{Pic}(X)$, the Seidel-Thomas twist $\mathsf{T}_{\mathscr{O}_X}$, or the duality functor $\mathbb{D} = \mathsf{Hom}(-,\mathscr{O}_X)$ on the right hand side?
		\item What are the images of "natural" objects like $\mathscr{O}_X$ or the residue fields $\kappa(x)$ for $x \in X$ on the right hand side?
	\end{enumerate}
	In this article we will describe the following answers:
	\begin{enumerate}
		\item By a result of Ballard, Favero and Katzarkov \cite{BFK11} it is known that $\Phi \circ \mathsf{T}_{\mathscr{O}_X} \circ (\mathscr{O}_X(1) \otimes -) \cong (1) \circ \Phi$. Using this we can describe the action of the whole Picard group of $X$ if the hypersurface $X$ is irreducible of dimension bigger than $2$. Furthermore we prove that the autoequivalence $\mathbb{D}$ corresponds to the composition $(-)^t \circ \mathsf{T}_{\Phi(\mathscr{O}_X)}$.
		\item We explain how to solve question two using a computer algebra system such as SINGULAR \cite{SINGULAR}. In the case of the structure sheaf it boils down to computing the "2-periodic" part of a minimal graded projective resolution of the residue field $K$ over the ring $K[X_0, \hdots, X_n]/(f)$ and similarly in the case of a rational point $p=[p_0:\hdots:p_n] \in X$ $\Phi(\kappa(p))$ (let us assume $p_i=1$) can easily be computed from finding the "2-periodic" part of a minimal graded projective resolution of the module $K[X_0, \hdots, X_n]/(X_0-p_0X_i,\hdots, X_n-p_nX_i)$.
	\end{enumerate}
	As an application we calculate the matrix factorisations of the rank one Maximal Cohen-Macaulay (MCM) modules over the complete local ring $K[[X,Y,Z]]/(f)$ for an algebraically closed field $K$ of arbitrary characteristic and $f$ a smooth Weierstraß polynomial. Kahn shows that this is a representation-tame problem and even described its Auslander-Reiten quiver in \cite{Kahn89} using Atiyah's classification of vector bundles on an elliptic curve \cite{Atiyah57}. In particular he proves that there are three families of rank one MCM modules over such a ring. But up to now, the concrete matrix factorisations for these modules were unknown. We will show that they are given by the following theorem, where $T=K[[X,Y,Z]]$, $S=T/(Y^2Z-X^3-aXZ^2-bZ^3)$, $E=\mathsf{Proj}\Big(K[X,Y,Z]/(Y^2Z-X^3-aXZ^2-bZ^3)\Big)$, $e=[0:1:0]$, $K$ is an algebraically closed field of arbitrary characteristic and $-4a^3-27b^2 \neq 0$. Here the restriction to simplified Weierstraß equations is unnecessary and only done as to achieve a nicer looking result: All of the computations presented in Chapter 2 can also be carried out in the case of an arbitrary Weierstraß cubic without the need for any additional arguments.
\begin{theorem}
	Let $P_E(\lambda,\mu)=-X^2-\lambda XZ-(a+\lambda^2)Z^2$. Then the following matrix factorisations are mutually non-isomorphic and describe all indecomposable rank one Maximal Cohen-Macaulay $S$-modules, where $[\lambda,\mu,1]$ runs through all rational points of $E-{e}$:
	\[
		T^2 \xrightarrow{\begin{pmatrix} P_E(\lambda,\mu) & -Z(Y + \mu Z)\\ Y - \mu Z & X- \lambda Z \end{pmatrix}}  T^2 \xrightarrow{\begin{pmatrix} X - \lambda Z & Z(Y + \mu Z) \\  \mu Z -Y & P_E(\lambda,\mu) \end{pmatrix}} T^2
	\]
	\\
	\[
		T^2\xrightarrow{\begin{pmatrix} -X^2-aZ^2 & bZ^2-Y^2\\ -Z & -X \end{pmatrix}} T^2 \xrightarrow{\begin{pmatrix} X & bZ^2-Y^2 \\ -Z & X^2+aZ^2 \end{pmatrix}}  T^2
	\]
	\\
	\[
		T^2 \xrightarrow{\begin{pmatrix}
			P_E(\lambda,\mu) & -Y-\mu Z \\
			-Z(Y-\mu Z) & \lambda Z - X
		\end{pmatrix}} T^2 \xrightarrow{\begin{pmatrix}
			X - \lambda Z & -Y - \mu Z \\
			-Z(Y - \mu Z) & -P_E(\lambda,\mu)
		\end{pmatrix}} T^2
	\]
	\\
	\[
		T^2 \xrightarrow{\begin{pmatrix}
			aZX - Y^2 + bZ^2 & -X \\
			-X^2 & -Z
		\end{pmatrix}} T^2 \xrightarrow{\begin{pmatrix}
			-Z & X \\
			X^2 & bZ^2-Y^2 + aXZ
		\end{pmatrix}} T^2
	\]
	\\
	\[
		T^3 \xrightarrow{\begin{pmatrix}
			P_E(\lambda,\mu) & -Z(Y + \mu Z) & \lambda \mu Z^2 + XY + \mu XZ + \lambda YZ \\
			-X(Y- \mu Z) & -X(X- \lambda Z) & -(a + \lambda^2) XZ +Y^2 -bZ^2\\
			-Z(Y- \mu Z) & -Z(X- \lambda Z) & X^2 + \lambda^2 Z^2
		\end{pmatrix}}
	\] 
	\[
		T^3 \xrightarrow{\begin{pmatrix}
			X- \lambda Z  & 0 & -Y - \mu Z\\
			\mu Z -Y & X + \lambda Z & (a + \lambda^2)Z \\
			0 & Z & -X
		\end{pmatrix}} T^3
	\]
	\\
	\[ T \xrightarrow{\begin{matrix} \; 1 \;\end{matrix}} T \xrightarrow{\begin{matrix} Y^2Z-X^3-aXZ^2-bZ^3\end{matrix}} T \]
\end{theorem}
	This extends earlier work of Laza, Pfister and Popescu who considered the Fermat polynomial $f=X^3+Y^3+Z^3$ in \cite{LPP02}.
	\\Unfortunately it seems out of reach to produce explicit matrix factorisations for all indecomposable MCM modules of higher rank. The best we can do is to describe a method how one can compute them with the help of a computer.
\begin{thanksto}
	I would like to thank my advisor Igor Burban for introducing me to Orlov's work, for setting interesting tasks and for his helpful advice. This work was supported by the DFG grant Bu-1866/3-1.
\end{thanksto}
\section{Computing with Orlov's Equivalence}
\subsection{Notations and Choice of the Equivalence}
	Let $A = K[X_0,\hdots, X_n]/(f)$ where $f$ is a homogeneous polynomial of degree $n+1$ (the grading on $K[X_0,\hdots,X_n]$ is given by $\mathsf{deg}(X_i)=1$ for all $0 \leq i \leq n$) and let $X= \mathsf{Proj}(A)$ be the associated projective hypersurface. Denote the smallest triangulated subcategory of $\mathcal{D}^b(\mathsf{gr}A)$ which contains the residue fields $K(j)$ for $j>-i$ by $\mathcal{S}_{<i}$. Similarly define $\mathcal{S}_{\geq i}$. Denote the smallest triangulated subcategory of $\mathcal{D}^b(\mathsf{gr}A)$ which contains the graded free modules $A(j)$ for $j>-i$ by $\mathcal{P}_{<i}$. Similarly define $\mathcal{P}_{\geq i}$. By results of Orlov \cite{Orlov09}, we have semiorthogonal decompositions $\mathcal{D}^b(\mathsf{gr}A) = \langle \mathcal{S}_{<i}, \mathcal{D}_i, \mathcal{S}_{\geq i} \rangle$, $\mathcal{D}^b(\mathsf{gr}A) = \langle \mathcal{P}_{\geq i}, \mathcal{T}_i, \mathcal{P}_{< i} \rangle$ and an equality $\mathcal{D}_i = \mathcal{T}_i$ for all $i \in \mathbb{Z}$.
	\\Furthermore we have the following commutative diagram of triangulated categories and exact functors
\begin{equation}
	\label{diag}
	\begin{xy}
		\xymatrix{
		& \mathcal{D}^b(\mathsf{gr}A) \ar[dl]_{\widetilde{(-)}} \ar[d] \ar[rd]^{\gamma_i} &  \mathcal{D}^b(\mathsf{gr}A) \ar[d]^{\delta_i} & \mathsf{MF}(f) \ar[l]_-{\mathsf{cok}} \ar[d]^{p} \\
		\mathcal{D}^b(\mathsf{Coh}(X)) & \mathcal{D}^b(\mathsf{qgr}A) \ar[l]_-{\simeq} \ar[r]^-{\simeq} & \mathcal{D}_i = \mathcal{T}_i & \underline{\mathsf{MF}}(f) \ar[l]^-{\mathsf{cok}_i}_-{\simeq}
		}
	\end{xy}
\end{equation}
where $\mathcal{D}^b(\mathsf{qgr}A) = \mathcal{D}^b(\mathsf{gr}A)/\mathcal{D}^b(\mathsf{tors}A)$, where $\widetilde{(-)}$ denotes Serre's functor (cf. \cite{Serre55}) and where the unnamed functor and $p$ are quotient functors. The lower left functor is an equivalence by results of Miyachi \cite{Miyachi91}. The functors $\gamma_i$ are given as the composites $\mathcal{D}^b(\mathsf{gr}A) \xrightarrow{\mathsf{tr}_{\geq i}} \mathcal{D}^b(\mathsf{gr}A_{\geq i}) \rightarrow \mathcal{D}_i$, where the unnamed arrow is the left adjoint to the inclusion of $\mathcal{D}_i$ in $\mathcal{D}^b(\mathsf{gr}A_{\geq i})$ and similarly for the functors $\delta_i$. By the theory of semiorthogonal decompositions (see for example \cite{Bondal89} or \cite{BK89}) these are quotient functors. Using that the duality $\mathsf{RHom}_{gr}(-,A): \mathcal{D}^b(\mathsf{gr}A) \rightarrow \mathcal{D}^b(\mathsf{gr}A)$ sends semiorthogonal decompositions to semiorthogonal decompositions, but exchanges the order, and the proof of Lemma 2.3 in \cite{Orlov09}, we can give an explicit description of $\gamma_i$ as the composite
$$\mathcal{D}^b(\mathsf{gr}A) \xrightarrow{\mathsf{tr}_{\geq i}} \mathcal{D}^b(\mathsf{gr}A) \xrightarrow{\mathsf{RHom}(-,A)} \mathcal{D}^b(\mathsf{gr}A) \xrightarrow{\mathsf{tr}_{\geq i-1}} \mathcal{D}^b(\mathsf{gr}A) \xrightarrow{\mathsf{RHom}(-,A)} \mathcal{D}^b(\mathsf{gr}A)$$
where $\mathsf{tr}_{\geq i}$ is the exact functor which sends a graded $A$-module $M$ to the module $\mathsf{tr}_{\geq i}(M)$ which is defined as $\mathsf{tr}_{\geq i}(M)_j := \begin{cases} M_j & j \geq i \\ 0 & else. \end{cases}$
\\Define $\Phi_i : \mathcal{D}^b(\mathsf{Coh}(X)) \rightarrow \underline{\mathsf{MF}}(f)$ to be the composite of the exact equivalences in the lower row of the diagram above. The different choices of the integer $i$ are related as follows:
\begin{lem}
	\label{lem-nochoice}
	We have $(1) \circ \Phi_i \circ (\mathscr{O}_X(-1) \otimes -) \cong \Phi_{i-1}$ for all $i \in \mathbb{Z}$.
\end{lem}
\begin{proof}
	Since $(1): \mathcal{D}^b(\mathsf{gr}A) \rightarrow \mathcal{D}^b(\mathsf{gr}A)$ restricts to an equivalence between $\mathcal{D}^b(\mathsf{gr}A_{\geq i})$ and $\mathcal{D}^b(\mathsf{gr}A_{\geq i-1})$ sending $\mathcal{P}_{\geq i}$ to $\mathcal{P}_{\geq i-1}$ and $\mathcal{S}_{\geq i}$ to $\mathcal{S}_{\geq i-1}$ , and $(1) \circ \mathsf{tr}_{\geq i} \cong \mathsf{tr}_{\geq i-1} \circ (1)$, we conclude
	\begin{align*}
		(1) \circ \gamma_i &\cong \gamma_{i-1} \circ (1) \\
		(1) \circ \delta_i &\cong \delta_{i-1} \circ (1).
	\end{align*}
	Also, we have a diagram of categories and exact functors for all $i \in \mathbb{Z}$ where each square except possibly for the trapezium in the middle commutes (at least up to natural isomorphism):
	$$
		\begin{xy}
			\xymatrix{
				\mathcal{D}^b(\mathsf{gr}A) \ar[ddd]^{(1)} \ar[r]^{\gamma_i} & \mathcal{D}_i = \mathcal{T}_i \ar[ddd]^{(1)}& \mathcal{D}^b(\mathsf{gr}A) \ar[l]_{\delta_i} \ar[ddd]^{(1)} && \mathsf{MF}(f) \ar[ll]_{\mathsf{cok}} \ar[ddd]^{(1)} \ar[dl]_{p}\\
				&&& \underline{\mathsf{MF}}(f) \ar[ull] \ar[d]^{(1)}\\
				&&& \underline{\mathsf{MF}}(f)\ar[dll] \\
				\mathcal{D}^b(\mathsf{gr}A) \ar[r]^-{\gamma_{i-1}} & \mathcal{D}_{i-1} = \mathcal{T}_{i-1} & \mathcal{D}^b(\mathsf{gr}A) \ar[l]_-{\delta_{i-1}} && \mathsf{MF}(f) \ar[ll]_{\mathsf{cok}} \ar[ul]^{p}
		}
		\end{xy}
	$$
	But because the functor $p$ is a localisation functor, the trapezium commutes, too.
	\\The proof is completed by noting that the quotient functor $\widetilde{(-)}$ and the equivalence $(1)$ commute:
	\begin{align*}
		(1) \circ \Phi_i \circ (\mathscr{O}_X(-1) \otimes -) \circ \widetilde{(-)}
		&\cong (1) \circ \Phi_i \circ \widetilde{(-)} \circ (-1) \\
		&\cong (1) \circ \mathsf{cok}_i^{-1} \circ \gamma_i \circ (-1) \\
		&\cong \mathsf{cok}_{i-1}^{-1} \circ (1) \circ \gamma_i \circ (-1) \\
		&\cong \mathsf{cok}_{i-1}^{-1} \circ \gamma_{i-1} \circ (1) \circ (-1) \\
		&\cong \mathsf{cok}_{i-1}^{-1} \circ \gamma_{i-1} \\
		&\cong \Phi_{i-1} \circ \widetilde{(-)}.
	\end{align*}
\end{proof}
	We now choose the equivalence $$\Phi := \Phi_1 : \mathcal{D}^b(\mathsf{Coh}(X)) \rightarrow \underline{\mathsf{MF}}(f)$$ to be the one which we will consider in the rest of this article. By Lemma \ref{lem-nochoice} this choice does not really effect the results we obtain about computing with $\Phi$.
\subsection{General Strategy of Computation}
Considering the commutative diagram (\ref{diag}) of the last section, we see that we can calculate $\Phi(\widetilde{C})$, where $C \in \mathcal{D}^b(\mathsf{gr}A)$,  as the preimage of $\gamma_1(C)$ under $\mathsf{cok}_i$ in $\underline{\mathsf{MF}}(f)$. Therefore our calculation can be split up into three steps:
\begin{enumerate}
	\item For an object $\mathcal{F} \in \mathcal{D}^b(\mathsf{Coh}(X))$ we first have to find a preimage $C$ under Serre's functor $\widetilde{(-)}: \mathcal{D}^b(\mathsf{gr}A) \rightarrow \mathcal{D}^b(\mathsf{Coh}(X))$. By results of Serre (see \cite{Serre55}) one can usually do this by calculating sheaf cohomology of $\mathcal{F}(i)$ for all $i \in \mathbb{Z}$. For certain classes of sheaves it is quite easy to guess the correct preimage, so this step is not much of a problem in practice.
	\item Since we know that $\gamma_1$ is given as the composite
$$\mathcal{D}^b(\mathsf{gr}A) \xrightarrow{\mathsf{tr}_{\geq 1}} \mathcal{D}^b(\mathsf{gr}A) \xrightarrow{\mathsf{RHom}(-,A)} \mathcal{D}^b(\mathsf{gr}A) \xrightarrow{\mathsf{tr}_{\geq 0}} \mathcal{D}^b(\mathsf{gr}A) \xrightarrow{\mathsf{RHom}(-,A)} \mathcal{D}^b(\mathsf{gr}A)$$
		we can actually reasonably hope to calculate $\gamma_1(C)$, since at worst it amounts to calculating two projective resolutions (possibly of honest complexes).
	\item For the third part, we have to remind ourselves of Buchweitz-Orlov's proof (see \cite{Buchweitz86} and \cite{Orlov09}) that the functor $\mathsf{cok}_i: \underline{\mathsf{MF}}(f) \rightarrow \mathcal{T}_i$ is essentially surjective, which works as follows:
		\\It suffices to check essential surjectivity on the images of the graded modules, since $cok_i$ is an exact functor of triangulated categories. Take a projective resolution of such a module. Then the depth lemma implies that an $n$th syzygy in this resolution will be the cokernel of a matrix factorisation of $f$. It thus only remains to shift accordingly to find a matrix factorisation corresponding to the module.
		\\Of course one can also filter a bounded complex by its cohomologies, calculate their matrix factorisations, find the correct morphisms between them and calculate these cones, but this seems to difficult in practice. Therefore this second step is only available if the cohomology of $\gamma_1(C)$ is concentrated in a single degree. Fortunately this is often the case as we will see in the following.
\end{enumerate}
\subsection{Computing $\Phi(\mathscr{O}_X)$ and $\Phi(\kappa(p))$}
Applying the strategy described in the last section to the structure sheaf $\mathscr{O}_X$ we are forced to calculate $\gamma_1(A)$. This is done by the following lemma.
\begin{lem}
	$\gamma_1(A) \cong A_{\geq 1}$, where $A_{\geq 1}$ denotes the irrelevant ideal of $A$.
\end{lem}
\begin{proof}
	Since $\mathsf{tr}_{\geq 1}(A) = A_{\geq 1}$ by definition and since $\mathsf{RHom}_{gr}(-,A)$ is self-inverse, it is sufficient to prove that $\mathsf{RHom}_{gr}(A_{\geq 1},A)$ is concentrated in degrees greater or equal to zero. This can be checked after applying cohomology. Using the long exact sequence in cohomology associated to the short exact sequence $$ 0 \rightarrow A_{\geq 1} \rightarrow A \rightarrow K \rightarrow 0$$ we find  $\mathsf{Ext}^i_{gr}(A_{\geq 1},A) \cong \mathsf{Ext}^{i+1}_{gr}(K,A)$ for all $i \geq 1$ and an exact sequence $A \cong \mathsf{Hom}_{gr}(A,A) \rightarrow \mathsf{Hom}_{gr}(A_{\geq 1},A) \rightarrow \mathsf{Ext}^1_{gr}(K,A)$, so the result follows since $A$ is an AS-Gorenstein $K$-algebra of Gorenstein parameter $a=0$.
\end{proof}
We now take a rational point $p = [p_0,\hdots,p_n] \in X$ and assume $p_i = 1$. To apply our general strategy we need to find a preimage of $\kappa(p)$ under the functor $\tilde{(-)}: \mathcal{D}^b(\mathsf{gr}A) \rightarrow \mathcal{D}^b(\mathsf{Coh}(X))$. This is accomplished by the following well-known lemma:
\begin{lem}
	The quasi-coherent sheaf associated to the finitely generated graded $A$-module  $A/(X_0-p_0X_i,\hdots,X_n-p_nX_i)$ is isomorphic to $\kappa(p)$.
\end{lem}
\begin{proof}
	The restriction of the associated sheaf to $V_+(X_i)$ may be calculated as first taking $A/(X_0-p_0X_i,\hdots,X_n-p_nX_i)$ modulo $X_i$ and then applying Serre's functor. But since $A/(X_0-p_0X_i,\hdots,X_n-p_nX_i,X_i) = K$ this is mapped to zero in $\mathsf{Coh}(X)$, which implies that the support of the associated sheaf is concentrated in $D_+(X_i)$. Setting $X_i = 1$ it becomes obvious that the associated sheaf is one copy of the ground field concentrated at the rational point $[p_0,\hdots,p_n]$. 
\end{proof}
Let us denote the module $A/(X_0-p_0X_i,\hdots,X_n-p_nX_i)$ by $A_{p_0,\hdots,p_n}$. It will be our next aim to calculate its image under the quotient functor $\gamma_1:\mathcal{D}^b(\mathsf{gr}A) \rightarrow \mathcal{D}_1$. This is achieved by the next lemma with similar techniques as in the case of the structure sheaf.
\begin{lem}
	$\gamma_1(A_{p_0,\hdots,p_n}) \cong A_{p_0,\hdots,p_n}(-1)$.
\end{lem}
\begin{proof}
	First, let us remark that $\mathsf{tr}_{\geq 1}(A_{p_0,\hdots,p_n})=(A_{p_0,\hdots,p_n})_{\geq 1}$ is isomorphic to $A_{p_0,\hdots,p_n}(-1)$.	
	Second, there is a short exact sequence of finitely generated graded $A$-modules $$0 \rightarrow A_{p_0,\hdots,p_n}(-2) \xrightarrow{X_i} A_{p_0,\hdots,p_n}(-1) \rightarrow K(-1) \rightarrow 0. $$
	Applying the functor $\mathsf{RHom}_{gr}(-,A)$ to the resulting distinguished triangle and using that $A$ is an AS-Gorenstein algebra with Gorenstein parameter $a=0$ and Nakayama's Lemma we see that $\mathsf{RHom}_{gr}(A_{p_0,\hdots,p_n}(-1),A)_{\geq 0} \cong \mathsf{RHom}_{gr}(A_{p_0,\hdots,p_n}(-1),A)$ and hence the claim as $\mathsf{RHom}_{gr}(-,A)$ is self-dual.
\end{proof}
These lemmas allow us to give the following recipe for computing the matrix factorisations $\Phi(\mathscr{O}_X)$ and $\Phi\big(\kappa(p)\big)$ which can easily be used to perform such calculations with the help of a computer and a program such as SINGULAR(\cite{SINGULAR}):
\begin{enumerate}
	\item Calculate a (minimal) projective resolution $P^\bullet$ of $K$, respectively of $A_{p_0,\hdots,p_n}$.
	\item Calculate $C = \mathsf{cok}(d^{-n}: P^{-n-1} \rightarrow P^{-n})$, respectively $C = \mathsf{cok}(d^{-n+1}: P^{-n} \rightarrow P^{-n+1})$.
	\item Take a graded free $K[X_0,\hdots,X_n]$-module $Q^1$ and an isomorphism $Q^1/(f)Q^1 \cong P^{-n}$ and calculate the kernel $\alpha: Q^0 \hookrightarrow Q^1$ of the composite $Q^1 \rightarrow Q^1/(f)Q^1 \cong P^{-n} \rightarrow C$, respectively everything with $-n$ replaced by $-n+1$.
	\item Calculate the unique morphism $\beta: Q^1 \rightarrow Q^0$ such that $\beta \circ \alpha = f$.
	\item Apply the shift functor $[1]: \underline{\mathsf{MF}}(f) \rightarrow \underline{\mathsf{MF}}(f)$ $n$, respectively $n-1$, times to the matrix factorisation $Q^0 \xrightarrow{\alpha} Q^1 \xrightarrow{\beta} Q^0(d)$ to finish the calculation.
\end{enumerate}
\begin{rem}
	As we will see in the next section, $\Phi\big(\mathscr{O}_X(-1)\big)$ is given as the matrix factorisation $\Phi(\mathscr{O}_X)[2-n](-1)$, since a well-known computation yields $\mathsf{T}_{\mathscr{O}_X}(\mathscr{O}_X) = \mathscr{O}_X[2-n]$. Hence, once one has calculated $\Phi(\mathscr{O}_X)$, one already knows $\Phi\big(\mathscr{O}_X(-1)\big)$, too.
\end{rem}
\subsection{Applications to the Action of the Picard Group}
In this section, we fix a projective, irreducible hypersurface $X$ of degree $n+1$ in $\mathbb{P}^n_K$, say it is cut out by the homogeneous polynomial $f$. Then a result of Grothendieck (see \cite{SGA2} or \cite{Hartshorne70}) shows that $\mathsf{Pic}(X) = \mathbb{Z}$ generated by $\mathscr{O}_X(1)$ if $n \geq 4$ (this does not make any assumptions on the characteristic of $K$ or on the smoothness of $X$!). So -in a sense- we know the action of the whole Picard group on $\underline{\mathsf{MF}}(f)$ if we can describe the action of the very ample line bundle $\mathscr{O}_X(1)$. We will also describe the action of $\mathscr{O}_X(-1)$.
\\To do so, we need to introduce some more notation: Since $X$ is a hypersurface of degree $n+1$ $\mathscr{O}_X$ is a spherical object and so we have the Seidel-Thomas twist functor $T_{\mathscr{O}_X}$ available \cite{ST01}. By a result of Ballard, Favero and Katzarkov \cite{BFK11}, the composite functor $$\mathsf{T}_{\mathscr{O}_X} \circ (\mathscr{O}_X(1) \otimes -): \mathcal{D}^b(\mathsf{Coh}(X)) \rightarrow \mathcal{D}^b(\mathsf{Coh}(X))$$ corresponds to the autoequivalence $$(1): \underline{\mathsf{MF}}(f) \rightarrow \underline{\mathsf{MF}}(f)$$ under the equivalence $\Phi$ which we have described in a previous section (they work under the assumptions $\mathsf{char}(K)=0$ and $X$ smooth, but these are not needed in their proof of the above result).
\\Therefore we have the following isomorphisms of functors:
\begin{align*}	
	\Phi \circ (\mathscr{O}_X(1) \otimes -) \circ \Phi^{-1} &\cong \mathsf{T}_{\Phi(\mathscr{O}_X)}^{-1} \circ (1) \\
	\Phi \circ (\mathscr{O}_X(-1) \otimes -) \circ \Phi^{-1} &\cong (-1) \circ \mathsf{T}_{\Phi(\mathscr{O}_X)}.
\end{align*}
Given a matrix factorisation $P^0 \xrightarrow{\alpha} P^1 \xrightarrow{\beta} P^0(d)$ (also denoted by $(\alpha,\beta)$ for short) we therefore have the following recipes for computing the action of $\mathscr{O}_X(1) \otimes -$ and $\mathscr{O}_X(-1) \otimes -$ on it. The action of $\mathscr{O}_X(-1)$ is given by:
\begin{enumerate}
	\item Calculate $K$-bases $\{f_{1,i},\hdots,f_{n_i,i}\}$, $i \in \mathbb{Z}$, of all the$\mathsf{Hom}$-spaces \linebreak $\mathsf{Hom}_{\underline{\mathsf{MF}}(f)}\Big(\Phi(\mathscr{O}_X)[i], (\alpha,\beta)\Big)$ (only finitely many of these $K$-vector spaces will be non-zero).
	\item Calculate the cone $(\gamma,\delta)$ of the morphism $\bigoplus_{i \in \mathbb{Z}}\Phi(\mathscr{O}_X)^{n_i}[i] \rightarrow (\alpha,\beta)$ which on the summand corresponding to $s,i$ is given by the morphism $f_{s,i}$.
	\item Apply the functor $(-1)$ to $(\gamma,\delta)$.
\end{enumerate}
The action of $\mathscr{O}_X(1)$ is given by:
\begin{enumerate}
	\item Let $(\gamma,\delta) = (\alpha,\beta)(1)$.
	\item Calculate $K$-bases $\{g_{1,i},\hdots,g_{n_i,i}\}$, $i \in \mathbb{Z}$, of all the $\mathsf{Hom}$-spaces \linebreak $\mathsf{Hom}_{\underline{\mathsf{MF}}(f)}\Big((\gamma,\delta),\Phi(\mathscr{O}_X)[i]\Big)$ (only finitely many of these $K$-vector spaces will be non-zero).
	\item Calculate the cone of the morphism $(\gamma,\delta) \rightarrow \bigoplus_{i \in \mathbb{Z}}\Phi(\mathscr{O}_X)^{n_i}[i]$ which on the factor corresponding to $s,i$ is given by the morphism $g_{s,i}$.
\end{enumerate}
\begin{rem}
	In general it will be difficult to predict for which shifts the corresponding $\mathsf{Hom}$-space will be non-zero. However if $X$ is a smooth elliptic curve and $(\alpha, \beta)$ is an indecomposable matrix factorisation, at most two of the groups $\mathsf{Hom}\Big(\Phi(\mathscr{O}_X)[i], (\alpha, \beta)\Big)$ will be nonzero (and in almost all cases, it will be only one) and they will be in neighbouring degrees since an indecomposable object of $\mathcal{D}^b(\mathsf{Coh}(X))$ has to have cohomology concentrated in one (cohomological) degree in this case. By Serre duality, at most two of the groups $\mathsf{Hom}\Big((\alpha,\beta), \Phi(\mathscr{O}_X)[i]\Big)$ will be non-zero, too.
\end{rem}
\subsection{The Action of the Duality Functor $\mathbb{D}$}
In the set-up we are considering $\mathscr{O}_X$ is a dualising complex in the sense of Grothendieck \cite{Hartshorne66} and one can wonder what autoequivalence of $\underline{\mathsf{MF}}(f)$ corresponds to the functor $\mathbb{D} := \mathsf{R}Hom(-,\mathscr{O}_X)$. It is natural to expect that the functor $(-)^t: \mathsf{MF}(f) \rightarrow \mathsf{MF}(f)$ which sends a matrix factorisation $(\alpha, \beta)$ to its transpose $\big(\alpha^t(-2d), \beta^t(-d)\big)$ (or rather what it induces on the stable category) will have something to do with it, but as in the case of the functor $\mathscr{O}_X(1) \otimes -$ it turns out that there is a "correction" term in form of a twist functor. The precise statement is given by the following
\begin{prop}
\label{prop-duality}
	There is an natural isomorphism of functors $$\Phi \circ \mathbb{D} \cong (-)^t \circ \mathsf{T}_{\Phi(\mathscr{O}_X)} \circ \Phi.$$
\end{prop}
Before proving this statement, we need two preparatory lemmas:
\begin{lem}
	\label{lemma-com1}
	The diagram
	\[
		\begin{xy}
			\xymatrix{
				\mathcal{D}^b(\mathsf{gr}A) \ar[d]_{\mathsf{RHom}_{gr}(-,A)} \ar[r]^-{\widetilde{(-)}} & \mathcal{D}^b(\mathsf{Coh}(X)) \ar[d]^{\mathbb{D}} \\
				\mathcal{D}^b(\mathsf{gr}A) \ar[r]^-{\widetilde{(-)}} & \mathcal{D}^b(\mathsf{Coh}(X))
			}
		\end{xy}
	\]
	commutes (up to natural isomorphism).
\end{lem}
\begin{proof}
	First of all, we will show that for a (not necessarily finitely generated) graded $A$-module $I$, which is an injective object in the category of (all) graded $A$-modules, $\widetilde{I}$ is an injective quasi-coherent sheaf on the scheme $X$.
	\\Indeed, by Lemma 2.1.3 of \cite{Conrad00} the injective quasi-coherent sheaves on any locally noetherian scheme are precisely the injective objects of the category $\mathscr{Mod}(X)$ of $\mathscr{O}_X$-modules which happen to be quasi-coherent, and the latter condition is local on $X$. Therefore it is sufficient to establish the homogeneous localisation of a graded injective modules stays injective for an $\mathbb{N}$-graded $K$-algebra generated in degree one. This is done in two steps:
	\begin{enumerate}
		\item On a noetherian $\mathbb{Z}$-graded ring any localisation at homogeneous elements preserves injectivity (see for example \cite{FF74} or \cite{BS98}).
		\item If $I$ is a graded injective $R=\oplus_{i \in \mathbb{Z}}R_i$-module, then $I_0$ is an injective $R_0$-module if $R$ is strongly graded by Chapter 2 of \cite{NvO04}. In particular this is the case for the localisation of an $\mathbb{N}$-graded ring generated in degree one at a homogeneous element of degree one, because such a ring is strongly graded.
	\end{enumerate}
	Secondly, for any finitely generated graded $A$-module (here we resume the set-up above, so $A$ is again $\mathbb{N}$-graded) $M$ and any graded $A$-module $N$ there is a natural isomorphism $\widetilde{\mathsf{Hom}_{gr}(M,N)} \cong Hom(\widetilde{M},\widetilde{N})$. This extends to a natural transformation of $Hom$-complexes if the first component is a bounded complex of finitely generated graded $A$-modules and the second component is a bounded complex of graded $A$-modules. Because it is an isomorphism if both are concentrated in a single degree, it will be a quasi-isomorphism in general, hence give a natural isomorphism when considered as a natural transformation of exact functors between the derived categories.
	\\Putting these two results together gives the required statement.
\end{proof}
\begin{lem}
	\label{lemma-com2}	
	The diagram
	\[
		\begin{xy}
			\xymatrix{
				\mathsf{MF}(f) \ar[d]_{(-)^t} \ar[r]^-{\mathsf{cok}} & \mathcal{D}^b(\mathsf{gr}A) \ar[d]^{\mathsf{RHom}_{gr}(-,A)} \\
				\mathsf{MF}(f) \ar[r]^-{\mathsf{cok}} & \mathcal{D}^b(\mathsf{gr}A)
			}
		\end{xy}
	\]
	commutes (up to natural isomorphism).
\end{lem}
\begin{proof}
	We will only consider the statement on objects, but the same reasoning allows us to treat morphisms, too.
	\\Let $P^0 \xrightarrow{\alpha} P^1 \xrightarrow{\beta} P^0(d)$ be a graded matrix factorisation of $f$. The composition around the upper right corner, sends it to $\mathsf{RHom}_{gr}(\mathsf{cok}(\beta),A)$, which is also given as the cokernel of $\beta^t(-d)$ because $\mathsf{RHom}_{gr}(-,A)$ and $\mathsf{Hom}_{gr}(-,A)$ agree on the graded modules which are arbitrary high syzygies. By definition, this is what $(\alpha,\beta)$ is mapped to under the composition around the lower left corner.
\end{proof}
Let us now prove Proposition \ref{prop-duality}.
\begin{proof}
	By Lemmas \ref{lemma-com1} and \ref{lemma-com2} there is a commutative diagram
	\[
		\begin{xy}
			\xymatrix{
				\mathcal{D}^b(\mathsf{Coh}(X)) \ar[d]^{\mathbb{D}} & \ar[l]_-{\widetilde{(-)}} \mathcal{D}^b(\mathsf{gr}A) \ar[d]^{\mathsf{RHom}_{gr}(-,A)} \ar[r]^-{\delta_0} & \mathcal{D}_0 = \mathcal{T}_0 \ar[d]^{\mathsf{RHom}_{gr}(-,A)} & \ar[l]_-{\gamma_0} \mathcal{D}^b(\mathsf{gr}A) \ar[d]^{\mathsf{RHom}_{gr}(-,A)} & \ar[l]_-{\mathsf{cok}} \mathsf{MF}(f) \ar[d]^{(-)^t} \\
				\mathcal{D}^b(\mathsf{Coh}(X)) & \ar[l]_-{\widetilde{(-)}} \mathcal{D}^b(\mathsf{gr}A) \ar[r]^-{\delta_1} & \mathcal{D}_1 = \mathcal{T}_1 & \ar[l]_-{\gamma_1} \mathcal{D}^b(\mathsf{gr}A) & \ar[l]_-{\mathsf{cok}} \mathsf{MF}(f).
			}
		\end{xy}
	\]
	Hence
		\begin{align*}				
			\Phi_1 \circ \mathbb{D} &\cong (-)^t \circ \Phi_0 \\
			&\cong (-)^t \circ (1) \circ \Phi_1 \circ (\mathscr{O}_X(-1) \otimes -) \\
			&\cong (-)^t \circ \Phi_1 \circ \mathsf{T}_{\mathscr{O}_X} \\
			&\cong (-)^t \circ \mathsf{T}_{\Phi_1(\mathscr{O}_X)} \circ \Phi_1.
		\end{align*}
\end{proof}
\section{Maximal Cohen Macaulay Modules on Cones over Elliptic Curves}
\subsection{Computations with Elliptic Curves}
Let $E = \mathsf{Proj}\Big(K[X,Y,Z]/(Y^2Z-X^3-aXZ^2-bZ^3)\Big)$ be an irreducible genus one curve, where $a,b \in K$. Let us denote the polynomial ring $K[X,Y,Z]$ by $R$, the quotient $K[X,Y,Z]/(Y^2Z-X^3-aXZ^2-bZ^3)$ by $A$ and the point $[0,1,0]$ by $e$. We will describe formulas for the matrix factorisations corresponding to the rational points of $E$ using the methods of the last chapter. 
\\Let $[\lambda,\mu,1]$ be a rational point of $E$ (the case of $e$ has to be treated separately).
\begin{rem}
	The polynomial $Y^2Z-X^3-aXZ^2-bZ^3 -Z(Y^2-\mu ^2Z^2) \in K[X,Y,Z]$ can be written as $(X-\lambda Z) \cdot (-X^2 -\lambda XZ -(a + \lambda^2)Z^2)$.
\end{rem}
Denote the homogeneous polynomial $$\Big(Y^2Z-X^3-aXZ^2-bZ^3 -Z(Y^2-\mu ^2Z^2)\Big)/(X-\lambda Z)$$  (and also its image in $A=K[X,Y,Z]/(Y^2Z-X^3-aXZ^2-bZ^3)$) by $P_E(\lambda,\mu)$.
\begin{lem}
	A minimal graded projective resolution of $A_{\lambda,\mu,1}$ is given by
	\[
		\hdots \rightarrow A(-4)^2 \xrightarrow{\begin{pmatrix} P_E(\lambda,\mu) & -Z(Y + \mu Z)\\ Y - \mu Z & X - \lambda Z  \end{pmatrix}} A(-2) \oplus A(-3) \]
		\\
	\[
		\xrightarrow{\begin{pmatrix} X - \lambda Z  & Z(Y + \mu Z) \\  \mu Z - Y & P_E(\lambda,\mu) \end{pmatrix}} A(-1)^2 \xrightarrow{\begin{pmatrix} Y-\mu Z & X-\lambda Z\end{pmatrix}}  A \rightarrow  A_{\lambda,\mu,1} \rightarrow 0
	\]
	where $\hdots$ denotes repeating the two $2 \times 2$ matrices (and adjusting the gradings accordingly).
\end{lem}
\begin{proof}
	First of all, a direct computation shows that the above is a complex. Then we remark that the (graded) depth of $A_{\lambda,\mu,1}$ is one. So since $A$ is Cohen-Macaulay of dimension two, the first syzygy of $A_{\lambda,\mu,1}$ in a minimal projective resolution will have a "2-periodic" minimal graded projective resolution (up to shifts in the grading). Since one matrix in a matrix factorisation determines the other it suffices to show that the complex is exact at the first spot where the free module has rank two.
	\\For this, let $L = \mathsf{ker}\Big(\begin{pmatrix} Y-\mu Z & X-\lambda Z\end{pmatrix}\Big)$, a graded $A$-module concentrated in degrees greater or equal to one. Obviously the image of the first column $\begin{pmatrix} X-\lambda Z \\ \mu Z - Y \end{pmatrix}$ of the matrix is a non-zero element of $L/(X,Y,Z)L$. If we can show that it and the image of the second column $\begin{pmatrix} Z(Y+ \mu Z) \\ P_E(\lambda,\mu) \end{pmatrix}$ are linearly independent over $K$, we will be done, because by general considerations we know that the matrix has to be square, so $L/(X,Y,Z)L$ will be two-dimensional and the image of what we wrote down is precisely the kernel $L$.
	\\So let us assume there exists $c \in K$ such that
	$$ c \begin{pmatrix} X-\lambda Z \\ \mu Z - Y \end{pmatrix} + \begin{pmatrix}Z(Y+ \mu Z) \\ P_E(\lambda,\mu) \end{pmatrix} \in (X,Y,Z)L.$$
	For grading reasons $c = 0$, so we only have to show that $$ \begin{pmatrix} Z(Y+ \mu Z) \\ P_E(\lambda,\mu) \end{pmatrix} \notin (X,Y,Z)L.$$
	Since there are precisely two linear forms dividing $Z(Y + \mu Z)$ (up to units), both of which also divide $Z(Y^2-\mu ^2Z^2)$ and none of which divide $Y^2Z-X^3-aXZ^2-bZ^3$, the proof is finished.
\end{proof}
\begin{cor}
	\label{cor-MF of a point}
	The matrix factorisation corresponding to the rational point $[\lambda,\mu,1]$ under the equivalence $\Phi$ is given by 
	\[
		R(-3) \oplus R(-4) \xrightarrow{\begin{pmatrix} X - \lambda Z  & Z(Y + \mu Z) \\  \mu Z - Y & P_E(\lambda,\mu) \end{pmatrix}}
	\]
	\\
	\[
		R(-2)^2 \xrightarrow{\begin{pmatrix} P_E(\lambda,\mu) & -Z(Y + \mu Z)\\ Y - \mu Z & X- \lambda Z \end{pmatrix}} R \oplus R(-1). \]
\end{cor}
The same method shows
\begin{lem}
	A minimal graded projective resolution of $A_{0,1,0}$ is given by
	\[
		\hdots \rightarrow A(-4)^2 \xrightarrow{\begin{pmatrix} -X^2-aZ^2 & bZ^2-Y^2\\ -Z & -X \end{pmatrix}} A(-2) \oplus A(-3) \]
		\\
	\[
		\xrightarrow{\begin{pmatrix} X & bZ^2-Y^2 \\ -Z & X^2+aZ^2 \end{pmatrix}} A(-1)^2 \xrightarrow{\begin{pmatrix} Z & X\end{pmatrix}}  A \rightarrow  A_{0,1,0} \rightarrow 0
	\]
	where $\hdots$ denotes repeating the two $2\times2$ matrices (and adjusting the gradings accordingly).
\end{lem}
and thus we can write down the corresponding matrix factorisation.
\begin{cor}
	The matrix factorisation corresponding to the rational point $[0,1,0]$ under the equivalence $\Phi$ is given by
	\[
		R(-3) \oplus R(-4) \xrightarrow{\begin{pmatrix} X & bZ^2-Y^2 \\ -Z & X^2+aZ^2 \end{pmatrix}}  
	\]
	\[
		R(-2)^2 \xrightarrow{\begin{pmatrix} -X^2-aZ^2 & bZ^2-Y^2\\ -Z & -X \end{pmatrix}} R \oplus R(-1). \]
\end{cor}
Let us now also calculate $\Phi(\mathscr{O}_E)$. For this we first have to find a minimal projective resolution of $K$ (more precisely: of the irrelevant ideal $A_{\geq 1}$, but of course this amounts to the same work).
\begin{lem}
	\label{lemma-minresK}
	A minimal graded projective resolution of $K$ is given by
	\[
		\hdots \rightarrow A(-5)^3 \oplus A(-6) \xrightarrow{\begin{pmatrix} -bZ^2-aXZ & -YZ & -X^2 & aZ^2Y \\ Y & Z & 0 & -X^2-aZ^2 \\ -X & 0 & Z & YZ \\ 0 & -X & -Y & -bZ^2  \end{pmatrix}} 
	\]
	\\
	\[
		A(-3)\oplus A(-4)^3 \xrightarrow{\begin{pmatrix} Z & YZ & X^2 & 0 \\ -Y & -bZ^2 & aYZ & X^2+aZ^2 \\ X & 0 & -bZ^2-aXZ & -YZ \\ 0 & X & Y & Z  \end{pmatrix}} A(-2)^3\oplus A(-3)
	\]
	\\
	\[
		\xrightarrow{\begin{pmatrix} Y & Z & 0 & -X^2-aZ^2 \\ -X & 0 & Z & YZ \\ 0 &-X & -Y & -bZ^2  \end{pmatrix}} A(-1)^3
		\xrightarrow{\begin{pmatrix}X & Y & Z \end{pmatrix}} A \rightarrow K \rightarrow 0
	\] 
	where $\hdots$ denotes repeating the matrix factorisation and adjusting the degrees accordingly.
\end{lem}
\begin{proof}
	A direct calculation shows that the above is a complex and that the two $4\times4$ matrices give a matrix factorisation. Let us verify that the kernel $L$ of $\begin{pmatrix} X & Y & Z\end{pmatrix}$ is precisely the image of the incoming matrix:
	\\Let $fX + gY +hZ = e(ZY^2 - X^3 -aXZ^2- bZ^3)$ as elements of $K[X,Y,Z]$, where $f,g,h,e$ are homogeneous polynomials. Then $g$ may not contain a summand of the form $Y^n$, so - by adding certain multiples of $\begin{pmatrix} Y \\ -X \\ 0 \end{pmatrix}$ and $\begin{pmatrix} 0 \\ Z \\ -Y \end{pmatrix}$ - we arrive at another element of the kernel with the property that $g = 0$. But this in then in the kernel of the map $\begin{pmatrix} X & Z\end{pmatrix}$, so by the case of the rational point $[0,1,0]$ we may write it as a linear combination of $\begin{pmatrix} Z \\ 0 \\ -X \end{pmatrix}$ and $\begin{pmatrix} X^2+aZ^2 \\ 0 \\ bZ^2-Y^2 \end{pmatrix}$. To finish this part of the proof, it only remains to note that
	$$\begin{pmatrix} X^2+aZ^2 \\ 0 \\ bZ^2-Y^2 \end{pmatrix} = -\begin{pmatrix} -X^2-aZ^2 \\ YZ \\ -bZ^2 \end{pmatrix} +Y \begin{pmatrix}0 \\ Z \\ -Y  \end{pmatrix}.$$
	Therefore we have a surjection from the cokernel $L^\prime$ of 	\[
		 \begin{pmatrix} Z & YZ & X^2 & 0 \\ -Y & -bZ^2 & aYZ & X^2+aZ^2 \\ X & 0 & -bZ^2-aXZ & -YZ \\ 0 & X & Y & Z  \end{pmatrix}  \] to $L$ and since we already know that the two $4 \times 4$ matrices form a matrix factorisation we will be done if we can show that this is actually an isomorphism.
	\\We will achieve this by calculating and comparing the Hilbert series of both modules:
	\\Via the short exact sequence $$0 \rightarrow L \rightarrow A(-1)^3 \rightarrow A_{\geq 1} \rightarrow 0$$ we conclude that $\mathsf{dim}_K L_i = 3 \binom{i+1}{2} - 3 \binom{i-2}{2} - (\binom{i+2}{2} - \binom{i-1}{2}).$
	\\Via the short exact sequence $$0 \rightarrow R(-3) \oplus R(-4)^3 \rightarrow R(-2)^3 \oplus R(-3) \rightarrow L^\prime \rightarrow 0 $$
	defining $L^\prime$ we conclude that $\mathsf{dim}_K L^\prime_i = 3 \binom{i}{2} - 3 \binom{i-2}{2}$.
	\\Hence it remains to show that these two (finite!) numbers are equal for each $i$. A direct calculation shows 
	$$ 3 \binom{i+1}{2} - 3 \binom{i-2}{2} - \Big(\binom{i+2}{2} - \binom{i-1}{2}\Big) = 6i - 9 = 3 \binom{i}{2} - 3 \binom{i-2}{2}$$
	finishing the proof.
\end{proof}
\begin{cor}
	The graded matrix factorisation $\Phi(\mathscr{O}_E)$ is given by 
	\[
		R(-3) \oplus R(-4)^3 \xrightarrow{\begin{pmatrix} Z & YZ & X^2 & 0 \\ -Y & -bZ^2 & aYZ & X^2+aZ^2 \\ X & 0 & -bZ^2-aXZ & -YZ \\ 0 & X & Y & Z  \end{pmatrix}}
	\]
	\\
	\[
		R(-2)^3 \oplus R(-3) \xrightarrow{\begin{pmatrix} -bZ^2-aXZ & -YZ & -X^2 & aZ^2Y \\ Y & Z & 0 & -X^2-aZ^2 \\ -X & 0 & Z & YZ \\ 0 & -X & -Y & -bZ^2  \end{pmatrix}}  R \oplus R(-1)^3.
	\]
\end{cor}
\begin{rem}
	Let $S=K[[X,Y,Z]]/(Y^2Z-X^3-aXZ^2-bZ^3)$. By a result of Yoshino-Kawamoto \cite{YK88} Auslander's fundamental module $E$ (which is defined by a short exact sequence $0 \rightarrow S \rightarrow E \rightarrow S \rightarrow K \rightarrow 0$ representing a non-zero element of $\mathsf{Ext}^2(K,S)$) is given as a third syzygy of the residue field $K$ in the hypersurface case. Hence the above computation produces a matrix factorisation of the fundamental module. This will be useful later on, because the fundamental module controls the Auslander-Reiten sequences in the category of MCM $S$-modules.
\end{rem}
Given the calculations above and using some structure on the category $\mathcal{D}^b(\mathsf{Coh}(E))$ we can now calculate the images of several families of line bundles. 
\begin{lem}
	The matrix factorisation $\Phi\big(\mathscr{O}_E(-p)\big)$, where $p = [\lambda,\mu,1]$, is given by
	\[
		R(-4)^2 \xrightarrow{\begin{pmatrix} P_E(\lambda,\mu) & -Z(Y + \mu Z)\\ Y - \mu Z & X- \lambda Z \end{pmatrix}} \]
	\\
	\[
		 R(-2) \oplus R(-3) \xrightarrow{\begin{pmatrix} X - \lambda Z & Z(Y + \mu Z) \\  \mu Z -Y & P_E(\lambda,\mu) \end{pmatrix}} R(-1)^2.
	\]
	The matrix factorisation $\Phi\big(\mathscr{O}_E(-e)\big)$ is given by
	\[
		R(-4)^2 \xrightarrow{\begin{pmatrix} -X^2-aZ^2 & bZ^2-Y^2\\ -Z & -X \end{pmatrix}} \]
	\\	
	\[
		R(-2) \oplus R(-3) \xrightarrow{\begin{pmatrix} X & bZ^2-Y^2 \\ -Z & X^2+aZ^2 \end{pmatrix}}  R(-1)^2.
	\]
\end{lem}
\begin{proof}
	As remarked earlier, by a result of \cite{BFK11} we have an isomorphism of functors $\Phi \circ \mathsf{T}_{\mathscr{O}_E} \circ (\mathscr{O}_E(1) \otimes -) \cong (1) \circ \Phi$. Let $p \in E$ be a rational point (including $e$). Now the short exact sequence $$ 0 \rightarrow \mathscr{O}_E(-p) \rightarrow \mathscr{O}_E \rightarrow \kappa(p) \rightarrow 0$$ and the fact $\kappa(p) \otimes \mathscr{O}_E(1) \cong \kappa(p)$ imply $\Phi\big(\mathscr{O}_E(-p)\big) \cong \Phi\big(\kappa(p)\big)[-1](1)$. Writing this out we arrive at the claimed matrix factorisations.
\end{proof}
\begin{lem}
	The matrix factorisation $\Phi\big(\mathscr{O}_E(-e-p)\big)$ for a rational point $p = [\lambda,\mu,1]$ is given by 
	\[
		R(-5) \oplus R(-4) \xrightarrow{\begin{pmatrix}
			P_E(\lambda,\mu) & -Y-\mu Z \\
			-Z(Y-\mu Z) & \lambda Z - X
		\end{pmatrix}}
	\]	
	\\
	\[
		R(-3)^2 \xrightarrow{\begin{pmatrix}
			X - \lambda Z & -Y - \mu Z \\
			-Z(Y - \mu Z) & -P_E(\lambda,\mu)
		\end{pmatrix}} R(-2) \oplus R(-1).
	\]
	\\The matrix factorisation $\Phi\big(\mathscr{O}_E(-2e)\big)$ is given by 
	\[
		R(-5) \oplus R(-4) \xrightarrow{\begin{pmatrix}
			aZX - Y^2 + bZ^2 & -X \\
			-X^2 & -Z
		\end{pmatrix}}
	\]	
	\\
	\[
		R(-3)^2 \xrightarrow{\begin{pmatrix}
			-Z & X \\
			X^2 & bZ^2-Y^2 + aXZ
		\end{pmatrix}} R(-2) \oplus R(-1).
	\]
\end{lem}
\begin{proof}
	Let $p \in E$ be a rational point. Starting with the short exact sequence $0 \rightarrow \mathscr{O}_E(-e-p) \rightarrow \mathscr{O}_E(-e) \rightarrow \kappa(p) \rightarrow 0$ and using the fact that $\mathsf{Hom}\Big(\mathscr{O}_E(-e), \kappa(p)\Big)$ is a one-dimensional $K$-vector space we see that it is sufficient to find a generator of $\mathsf{Hom}\Big(\Phi\big(\mathscr{O}_E(-e)\big), \Phi\big(\kappa(p)\big)\Big)$, calculate its cone $C$ and apply the functor $[-1]$ to this cone.
	\\Let now $p = [\lambda,\mu,1]$. Then a direct calculation shows that the following diagram is commutative:
	\[
	\begin{xy}
		\xymatrix{
			R(-4)^2 \ar[rrrrr]^{\begin{pmatrix} -X^2-aZ^2 & bZ^2-Y^2 \\ -Z & -X  \end{pmatrix}} \ar[dddd]_{\begin{pmatrix} 0 & Y + \mu Z \\ 1 & \lambda  \end{pmatrix}} &&&&& R(-2) \oplus R(-3) \ar[dddd]^{\begin{pmatrix} 0 & -Y - \mu Z \\ 1 &\lambda X + \lambda^2 Z \end{pmatrix}} \\ \\ \\ \\
			R(-3) \oplus R(-4) \ar[rrrrr]^{\begin{pmatrix} X - \lambda Z  & Z(Y+ \mu Z) \\ \mu Z -Y & P_E(\lambda,\mu)  \end{pmatrix}} &&&&& R(-2)^2
		}
	\end{xy}
	\]
	Since the the vertical matrices contain entries which do not lie in the ideal $(X,Y,Z)$, this map cannot be homotopic to zero and so must be a generator of $\mathsf{Hom}\Big(\Phi\big(\mathscr{O}_E(-e)\big), \Phi\big(\kappa(p)\big)\Big)$.
	\\Applying elementary row and column transformations to the resulting cone
	\[
		R(-3) \oplus R(-4) \oplus R(-2) \oplus R(-3) \xrightarrow{\begin{pmatrix}
			X - \lambda Z  & Z(Y + \mu Z) & 0 & -Y- \mu Z \\
			\mu Z - Y & P_E(\lambda,\mu) & 1 & \lambda X + \lambda^2 Z \\
			0 & 0 & -X & -bZ^2+Y^2 \\
			0 & 0 & Z & -X^2-aZ^2
		\end{pmatrix}}
	\]
	\\
	\[
		R(-2)^2 \oplus R(-1)^2 \xrightarrow{\begin{pmatrix}
			P_E(\lambda,\mu) & -Z(Y + \mu Z) & 0 & Y +\mu Z \\
			Y - \mu Z & X - \lambda Z & 1 & \lambda \\
			0 & 0 & X^2+aZ^2 & Y^2 -bZ^2 \\
			0 & 0 & Z & X
		\end{pmatrix}} R \oplus R(-1) \oplus R(1) \oplus R
	\]
	we find the reduced matrix factorisation
	\[
		R(-3)^2 \xrightarrow{\begin{pmatrix}
			X - \lambda Z & -Y - \mu Z \\
			-Z(Y - \mu Z) & -P_E(\lambda,\mu)
		\end{pmatrix}} 
	\]
	\\
	\[
		R(-2) \oplus R(-1) \xrightarrow{\begin{pmatrix}
			P_E(\lambda,\mu) & -Y-\mu Z \\
			-Z(Y - \mu Z) & \lambda Z -X
		\end{pmatrix}} R^2.
	\]
	Applying the functor $[-1]$ we arrive at the expected matrix factorisation.
	\\ The same kind of argument works for the case of $\Phi\big(\mathscr{O}_E(-2e)\big)$ by using the commutative diagram
	\[
	\begin{xy}
		\xymatrix{
			R(-3)^2 \ar[rrrrr]^{\begin{pmatrix} -X^2-aZ^2 & bZ^2-Y^2 \\ -Z & -X  \end{pmatrix}} \ar[dddd]^{\begin{pmatrix} X & -aZ \\ 0 & -1  \end{pmatrix}} &&&&& R(-1) \oplus R(-2) \ar[dddd]^{\begin{pmatrix} -1 & aZ \\ 0 & X \end{pmatrix}} \\ \\ \\ \\
			R(-2) \oplus R(-3) \ar[rrrrr]^{\begin{pmatrix} X & bZ^2-Y^2 \\ -Z & X^2+aZ^2  \end{pmatrix}} &&&&& R(-1)^2.
		}
	\end{xy}
	\]
\end{proof}
\begin{lem}
	The matrix factorisation $\Phi\big(\mathscr{O}_E(-2e-p)\big)$ for $p = [\lambda,\mu,1]$ is given by
	\[
		R(-5)^3 \xrightarrow{\begin{pmatrix}
			P_E(\lambda,\mu) & -Z(Y + \mu Z) & \lambda \mu Z^2 + XY + \mu XZ + \lambda YZ \\
			-X(Y- \mu Z) & -X(X- \lambda Z) & -(a + \lambda^2) XZ +Y^2 -bZ^2\\
			-Z(Y- \mu Z) & -Z(X- \lambda Z) & X^2 + \lambda^2 Z^2
		\end{pmatrix}}
	\] 
	\[
		R(-3)^3 \xrightarrow{\begin{pmatrix}
			X- \lambda Z  & 0 & -Y - \mu Z\\
			\mu Z -Y & X + \lambda Z & (a + \lambda^2)Z \\
			0 & Z & -X
		\end{pmatrix}} R(-2)^3.
	\]	
\end{lem}
\begin{proof}
	The proof is very similar to the ones above. This time we use the short exact sequence $0 \rightarrow \mathscr{O}_E(-2e-p) \rightarrow \mathscr{O}_E(-2e) \rightarrow \kappa(p) \rightarrow 0$ and the commutative diagram
	\[
		\begin{xy}
			\xymatrix{
				R(-5) \oplus R(-4) \ar[rrrrr]^{\begin{pmatrix} aXZ -Y^2+bZ^2 & -X \\ -X^2 & -Z \end{pmatrix}} \ar[dddd]^{\begin{pmatrix} \lambda \mu Z^2+XY+ \mu XZ + \lambda YZ & 0 \\ \lambda^2 Z & 1	\end{pmatrix}} &&&&& R(-3)^2  \ar[dddd]^{\begin{pmatrix} 0 & -Y - \mu Z \\ X + \lambda Z & (a+ \lambda^2)Z \end{pmatrix}}
				\\ \\ \\ \\
				R(-3) \oplus R(-4) \ar[rrrrr]^{\begin{pmatrix} X-\lambda Z & Z(Y+\mu Z) \\ \mu Z - Y & P_E(\lambda,\mu) \end{pmatrix}} &&&&& R(-2)^2.
			}
		\end{xy}
	\]
\end{proof}
\begin{rem}
	There are isomorphisms $\mathsf{T}_{\mathscr{O}_E}\Big(\mathscr{O}_E(1) \otimes \mathscr{O}_E(-3e)\Big) \cong \mathsf{T}_{\mathscr{O}_{E}}(\mathscr{O}_E) \cong \mathscr{O}_E$, therefore the matrix factorisation $\Phi\big({\mathscr{O}_E(-3e)}\big)$ is given as a shift of the matrix factorisation $\Phi(\mathscr{O}_E)$. In particular it is a $4 \times 4$ matrix factorisation.
\end{rem}
We need one final lemma which states that there are not too many matrix factorisations of small rank in a sense to be made precise during the proof of the Theorem \ref{thm-main}.
\begin{lem}
	\label{lemma-too big}
	Let $\mathscr{L}$ be a line bundle on $E$ of degree $\mathsf{deg}(\mathscr{L}) \leq -4$. Then the size of the matrices of $\Phi(\mathscr{L})$ is bigger or equal to 4.
\end{lem}
\begin{proof}
	Let us only consider the case $\mathsf{deg}(\mathscr{L}) = -4$, the other ones being similar. The line bundle $\mathscr{L}$ can be written as $\mathscr{O}_E(-3e-p)$ for a rational point $p \in E$ and fits inside a short exact sequence $0 \rightarrow \mathscr{O}_E(-3e-p) \rightarrow \mathscr{O}_E(-3e) \rightarrow \kappa(p) \rightarrow 0$. Applying $\Phi$ to the corresponding distinguished triangle, we have to compute the cone of a non-zero morphism $\Phi(\mathscr{O}_E(-3e)) \rightarrow \Phi(\kappa(p))$. If we only write the corresponding graded free modules, the first matrix factorisation is given by $R(-4) \oplus R(-5)^3 \rightarrow R(-3)^3 \oplus R(-4) \rightarrow R(-1) \oplus R(-2)^3$, because $T_{\mathscr{O}_E}(\mathscr{O}_E(-3e) \otimes \mathscr{O}_E(1)) = \mathscr{O}_E$. According to Corollary \ref{cor-MF of a point} the second one is given by $R(-3) \oplus R(-4) \rightarrow R(-2)^2 \rightarrow R \oplus R(-1)$. Therefore there is at most one morphism of degree zero involved and the cone is a $5 \times 5$ or $6 \times 6$ matrix factorisation in its reduced form.
\end{proof}
\subsection{Classification of Rank One Maximal Cohen Macaulay Modules}
Let $f=Y^2Z-X^3-aXZ^2-bZ^3$, let $E$ be the elliptic curve defined by $f$ (this time we really want it to be smooth!), let $T=K[[X,Y,Z]]$ and let $S = T/(f)$. Here $K$ denotes an algebraically closed field of arbitrary characteristic.
\begin{thm}
	\label{thm-main}
	The following matrix factorisations are mutually non-iso\-mor\-phic and describe all indecomposable rank one Maximal Cohen-Macaulay $S$-modules, where $[\lambda,\mu,1]$ runs through all rational points of $E-{e}$:
	\[
		T^2 \xrightarrow{\begin{pmatrix} P_E(\lambda,\mu) & -Z(Y + \mu Z)\\ Y - \mu Z & X- \lambda Z \end{pmatrix}}  T^2 \xrightarrow{\begin{pmatrix} X - \lambda Z & Z(Y + \mu Z) \\  \mu Z -Y & P_E(\lambda,\mu) \end{pmatrix}} T^2
	\]
	\\
	\[
		T^2\xrightarrow{\begin{pmatrix} -X^2-aZ^2 & bZ^2-Y^2\\ -Z & -X \end{pmatrix}} T^2 \xrightarrow{\begin{pmatrix} X & bZ^2-Y^2 \\ -Z & X^2+aZ^2 \end{pmatrix}}  T^2
	\]
	\\
	\[
		T^2 \xrightarrow{\begin{pmatrix}
			P_E(\lambda,\mu) & -Y-\mu Z \\
			-Z(Y-\mu Z) & \lambda Z - X
		\end{pmatrix}} T^2 \xrightarrow{\begin{pmatrix}
			X - \lambda Z & -Y - \mu Z \\
			-Z(Y - \mu Z) & -P_E(\lambda,\mu)
		\end{pmatrix}} T^2
	\]
	\\
	\[
		T^2 \xrightarrow{\begin{pmatrix}
			aZX - Y^2 + bZ^2 & -X \\
			-X^2 & -Z
		\end{pmatrix}} T^2 \xrightarrow{\begin{pmatrix}
			-Z & X \\
			X^2 & bZ^2-Y^2 + aXZ
		\end{pmatrix}} T^2
	\]
	\\
	\[
		T^3 \xrightarrow{\begin{pmatrix}
			P_E(\lambda,\mu) & -Z(Y + \mu Z) & \lambda \mu Z^2 + XY + \mu XZ + \lambda YZ \\
			-X(Y- \mu Z) & -X(X- \lambda Z) & -(a + \lambda^2) XZ +Y^2 -bZ^2\\
			-Z(Y- \mu Z) & -Z(X- \lambda Z) & X^2 + \lambda^2 Z^2
		\end{pmatrix}}
	\] 
	\[
		T^3 \xrightarrow{\begin{pmatrix}
			X- \lambda Z  & 0 & -Y - \mu Z\\
			\mu Z -Y & X + \lambda Z & (a + \lambda^2)Z \\
			0 & Z & -X
		\end{pmatrix}} T^3
	\]	
	\\
	\[ T \xrightarrow{\begin{matrix} \; 1 \; \end{matrix}} T \xrightarrow{\begin{matrix}Y^2Z-X^3-aXZ^2-bZ^3\end{matrix}} T \]
\end{thm}
\begin{proof}
	By a result of Kahn (Proposition 5.23 in \cite{Kahn88}) any indecomposable MCM $S$-module is gradable, meaning it is the image of an indecomposable object of $\mathsf{MCM}(A)$ under the functor $\mathsf{gr}A \xrightarrow{\mathsf{forget}} A\mhyphen\mathsf{mod} \xrightarrow{\widehat{(-)}} S\mhyphen\mathsf{mod}$. Furthermore the images of two indecomposable graded MCM $A$-modules are isomorphic if and only if the original graded modules differ by some shift $(n)$, $n \in \mathbb{Z}$ (see Lemma 15.2 in \cite{Yoshino90}).
	\\Let $M$ be a graded indecomposable MCM $A$-module of rank one. By Corollary 1.3 of \cite{HK87} its matrix factorisation will be given by $1 \times 1$, $2 \times 2$ or $3 \times 3$ matrices. The first case clearly corresponds to the matrix factorisation $K[[X,Y,Z]] \xrightarrow{\quad 1 \quad} K[[X,Y,Z]] \xrightarrow{Y^2Z-X^3-aXZ^2-bZ^3} K[[X,Y,Z]].$
	\\By Lemma 2.34 of \cite{BD08} the determinant of the matrix giving a rank one MCM-module is $Y^2Z-X^3-aXZ^2-bZ^3$ and hence we may assume that the graded matrix factorisation of $M$ or its shift (which we also denote by $M$) is given by
	$$R(-4) \oplus R(-5) \rightarrow R(-3)^2 \rightarrow R(-1) \oplus R(-2) \rightarrow M \rightarrow 0$$ 
	or 
	$$R(-5)^3 \rightarrow R(-3)^3 \rightarrow R(-2)^3 \rightarrow M \rightarrow 0.$$
	Claim: $M \in \mathcal{T}_1$.
	\\Since $M \in \mathcal{D}^b(\mathsf{gr}A_{\geq 1})$ by construction, we need to show that there are no non-zero homomorphisms from $M$ into $A(-i)[n]$ for all $n \in \mathbb{Z}$ and all $i \geq 1$. For any $n \geq 1$ and any $i \in \mathbb{Z}$, $\mathsf{Ext}^n\Big(M,A(i)\Big) = 0$ because $M$ is a graded MCM-module and $A$ has finite injective dimension. For all $n < 0$ $\mathsf{Hom}\Big(M,A(i)[n]\Big) = 0$ in any case, so it remains to treat $\mathsf{Hom}\Big(M,A(-i)\Big)$ for $i \geq 1$. Using that the cokernel of $0 \rightarrow M \rightarrow A^k$ is also a graded MCM-module, where $k \in \{2,3\}$, we find a surjection $\mathsf{Hom}\Big(A^k,A(-i)\Big) \rightarrow \mathsf{Hom}\Big(M,A(-i)\Big)$. Since the former group is zero, so is the latter.
	\\Hence $\Phi^{-1}(M) = \widetilde{M}$. Therefore $\Phi^{-1}(M)$ is a vector bundle. By the results in Chapter 1.5 of \cite{BH93} the rank of $M$ as a matrix factorisation and the rank of $\widetilde{M}$ as a vector bundle agree, thus $\Phi^{-1}(M)$ is even a line bundle of degree $-9 \leq \mathsf{deg}(\widetilde{M}) \leq 0$. Furthermore Lemma \ref{lemma-too big} allows us to conclude that $\mathsf{deg}(\widetilde{M}) \geq -3$.
	\\Using the relation $\Phi \circ \mathsf{T}_{\mathscr{O}_E} \circ (\mathscr{O}_E(1) \otimes -) \cong (1) \circ \Phi$ once more and noticing that $T_{\mathscr{O}_E}$ acts as the identity functor on the non-trivial line bundles of degree zero, as well as the calculation of $\Phi(\mathscr{O}_E)$ as a $4 \times 4$ matrix factorisation, we see that we need only consider line bundles $\mathscr{L}$ of degree $-3 \leq \mathsf{deg}(\mathscr{L}) < 0$ excluding $\mathscr{O}_E(-3e)$. Their matrix factorisations have been calculated earlier in this chapter giving precisely the claimed answer.
	\\Finally, all these graded matrix factorisations are pairwise non-isomorphic, since their preimages under the equivalence $\Phi$ are. To finish the proof, we therefore only have to check that no shifts of the matrix factorisations in the $2 \times 2$ families are isomorphic. This is clear for grading reasons.
\end{proof}
\begin{rems}
	\begin{enumerate}
		\item If one wants to work with elliptic curves in Hesse form $X^3+Y^3+Z^3-\tau XYZ$ (as is done in the article \cite{LPP02}) one can use Nagell's algorithm (which is explained in Chapter 10 of \cite{Dolgachev03} for example) to find a projective transformation taking a cubic in Hesse form into its Weierstraß form (at least if the characteristic is different from two or three) and apply its inverse to the matrix factorisations of the previous theorem.
		\item One may wonder if the method of proof employed above can also be used for higher dimensional and/or singular hypersurfaces. Of course, we can still use the techniques of computation to produce families of matrix factorisations, but in general the completion functor $\mathsf{MCM}(A) \rightarrow \mathsf{MCM}(S)$ will not be dense (this is already the case for singular cubic curves) and we do not have as much control over line bundles as in the case of an elliptic curve. Therefore it doesn not seem likely that further such complete classification results can be achieved with this method.
	\end{enumerate}
\end{rems}
\subsection{Indecomposable Maximal Cohen-Macaulay Modules of Higher Rank}
After the explicit description of the rank one MCM modules, we will now describe a computer algebra based approach to calculating all indecomposable MCM modules over the ring $S=K[[X,Y,Z]]/(Y^2Z-X^3-aXZ^2-bZ^3)$ (still under the assumptions $K=\bar{K}$ of arbitrary characteristic and $4a^3-27b^2 \neq 0$). All of the tasks described next can be performed by a computer program such as SINGULAR \cite{SINGULAR}.
This is done in two steps as follows:
\\The (exact) category of $\mathsf{MCM}(S)$ of Maximal Cohen-Macaulay $S$-modules has Auslander-Reiten sequences which can be described explicitly as follows:
\\Denote by $E$ Auslander's fundamental module which is defined via an exact sequence
\[
	0 \rightarrow R \rightarrow E \rightarrow R \rightarrow k \rightarrow 0
\] corresponding to a non-zero element of $\mathsf{Ext}^2(K,S)$. If $M \neq S$ then the Auslander-Reiten sequence ending in $M$ is given by applying the functor $\mathsf{Hom}(\mathsf{Hom}(M,S),-)$ to the above exact sequence and is hence of the form
\[
	0 \rightarrow M \rightarrow \mathsf{Hom}(\mathsf{Hom}(M,S),E) \rightarrow M \rightarrow 0.
\]
This is explained in Chapter 11 of \cite{Yoshino90} for example. Therefore the middle term is computable using a computer, since we know a matrix factorisation of $E$ by the remark after Lemma \ref{lemma-minresK}. Furthermore, the middle term decomposes as a direct sum of at most two indecomposable MCM $S$-modules. The category $\mathsf{MCM}(A)$ has Auslander-Reiten sequences, too.
Since the completion functor $\mathsf{MCM}(A) \rightarrow \mathsf{MCM}(S)$ reflects exactness, preserves Auslander-Reiten sequences and induces an isomorphism $\mathsf{Ext}^2_{gr}(K,A) \cong \mathsf{Ext}^2(K,S)$ the Auslander-Reiten sequences in $\mathsf{MCM(}A)$ are of the form \[
	0 \rightarrow M \rightarrow \mathsf{Hom}_{gr}(\mathsf{Hom}_{gr}(M,A),E) \rightarrow M \rightarrow 0
\]
where now $E$ denotes the graded fundamental module and $M$ denotes a graded, non-free MCM $A$-module.
\\Moreover, Serre's functor $\widetilde{(-)}: \mathsf{gr}A \rightarrow \mathsf{Coh}(E)$ restricts to an equivalence $\mathsf{MCM}(A) \cong \mathsf{VB}(E)$, where $\mathsf{VB}(E)$ denotes the category of vector bundles on $E$. This is not an equivalence of exact categories, but since $\tilde{(-)}$ is an exact functor (of abelian categories), the set of short exact sequences of MCM-modules is mapped into the set of short exact sequences of vector bundles. Thus for any $M \neq A$ the image of the Auslander-Reiten sequence starting and ending in $M$ in the category of vector bundles is again an Auslander-Reiten sequence. The Auslander-Reiten quiver of the latter category decomposes as a disjoint union of tubes and the bottom of any tube is given by a vector bundle whose endomorphisms are just $K$. Since the dense completion functor $\mathsf{MCM}(A) \rightarrow \mathsf{MCM}(S)$ preserves Auslander-Reiten sequences, it suffices -as a second step- to be able to calculate all graded MCM $A$-module whose endomorphism rings are just $K$. As the property of having just $K$ as  endomorphism ring is preserved under passage to $\underline{\mathsf{MF}}(f)$ (for all matrix factorisations but the trivial one of rank one), we may restrict our attention to the latter category which is equivalent to $\mathcal{D}^b(\mathsf{Coh}(E))$ where it is known that any such object can be derived from $\{\kappa(p)\}_{p \in E}$ by applying the functors $[1]$, $\mathsf{T}_{\mathscr{O}_E}$ and $\mathsf{T}_{\kappa(e)} \cong \mathscr{O}_E(e)\otimes-$ (possibly several times), cf. \cite{Atiyah57} or \cite{LM00}. The latter functors (or more precisely: the corresponding functors on $\underline{\mathsf{MF}}(f)$) can also be computed using the help of a computer (because we have already computed the matrix factorisations of the objects along which we twist), thus we can compute any indecomposable matrix factorisation.

	$ \mathtt{Universität \; zu \; Köln, \;Mathematisches \; Institut, \; Weyertal \; 86\mhyphen90,}$\\ $\mathtt{ 50931 \; Köln, \; Germany}$
	\\$E \mhyphen mail \; address: \mathsf{lgalinat@math.uni\mhyphen koeln.de}$
\end{document}